\documentclass[10pt]{article}

\usepackage{amsmath}
\usepackage{amsthm}
\usepackage{amssymb}
\usepackage{fullpage}
\usepackage{pst-all}
\usepackage{graphicx}
\usepackage{subfigure}
\usepackage{overpic}

\newtheorem{theorem}{Theorem}[section]
\newtheorem{lemma}[theorem]{Lemma}

\newtheorem{proposition}[theorem]{Proposition}
\newtheorem{Remark}[theorem]{Remark}
\newenvironment{remark}{\begin{Remark}\rm}{\end{Remark}}
\numberwithin{equation}{section}

\newcommand{\eps}{\varepsilon}
\newcommand{\R}{\mathbb{R}}
\newcommand{\C}{\mathbb{C}}
\newcommand{\N}{\mathbb{N}}
\newcommand{\Z}{\mathbb{Z}}
\newcommand{\OO}{\mathcal{O}}
\renewcommand{\Im}{\mathop{\mathrm{Im}}}

\title{Limits of determinantal  processes near a tacnode}
\author{Alexei Borodin \footnotemark[1] \and Maurice Duits\footnotemark[2]}
\date{}

\begin{document}

\maketitle

\renewcommand{\thefootnote}{\fnsymbol{footnote}}
\footnotetext[1]{Department of Mathematics, California Institute of Technology, 1200 E. California Blvd, Pasadena, USA and IITP RAS, Moscow, Russia. email:
borodin\symbol{'100}caltech.edu.
This work was partially supported by NSF grant DMS-07070163. } \footnotetext[2]{Department of Mathematics, California Institute of Technology, 1200 E. California Blvd, Pasadena, USA. email:
mduits\symbol{'100}caltech.edu}

\begin{abstract}
We study a  Markov process on a system of interlacing particles. At large times the particles  fill a domain that depends on a parameter $\eps> 0$. The domain has two cusps, one pointing up and one pointing down. In the limit $\varepsilon\downarrow 0$  the cusps touch, thus forming a tacnode. The main result of the paper is a derivation of the local correlation kernel around the tacnode in the transition regime $\varepsilon \downarrow 0$. We also prove that the local process interpolates between the Pearcey process and   the GUE minor process.

\textbf{Keywords}: Determinantal point processes. Random growth. GUE minor process. Pearcey process.

\end{abstract}

\section{Introduction}
In a recent paper \cite{BorFer}  the authors introduced  a Markov process on a system of interlacing particles. This model contains many parameters, creating a rich pool of interesting particular examples, and at the same time it has an integrable structure that allows for explicit computations. It is the purpose of this paper, to study a special case of this model. The interest of this model lies in the fact that  for large time the particles will fill a domain that  has a tacnode on the boundary. A somewhat similar situation also occurs  in the case of non-intersecting Brownian paths with multiple sources and sinks \cite{AFvM}.  The local process at the tacnode in this model is not understood, although a conjecture  is given in \cite{AFvM}.  The integrability of the model we consider allows us  to compute the local process around the tacnode. This is the main result of this paper.

 We consider an evolution on particles that are placed on the grid
\begin{align}
\mathcal G =\left\{(x,m) \mid m=1,2,\ldots \quad x\in \Z+\frac{m+1}{2}\right\}.
\end{align}
 Hence, if $(x,m) \in \mathcal G$, then $x$ takes integer values for odd values of $m$ and half-integer values for even values of $m$. At each horizontal $m$-section  we  put $m$ particles and denote their horizontal coordinates by $x_k^m$ for $k=1,\ldots, m$.   The evolution is such that at each time the system of particles satisfies the interlacing condition
\begin{align}
x_{k-1}^{m}<x_{k-1}^{m-1}<x_k^{m} ,\qquad k=2,\ldots, m, \qquad m=2,\ldots
\end{align}
At time $t=0$ we put the particles at positions $x_k^m=-(m+1)/2+k$ as shown in Figure \ref{fig:particles}.

\begin{figure}[t]\begin{center}
\subfigure[]{\includegraphics[scale=0.35]{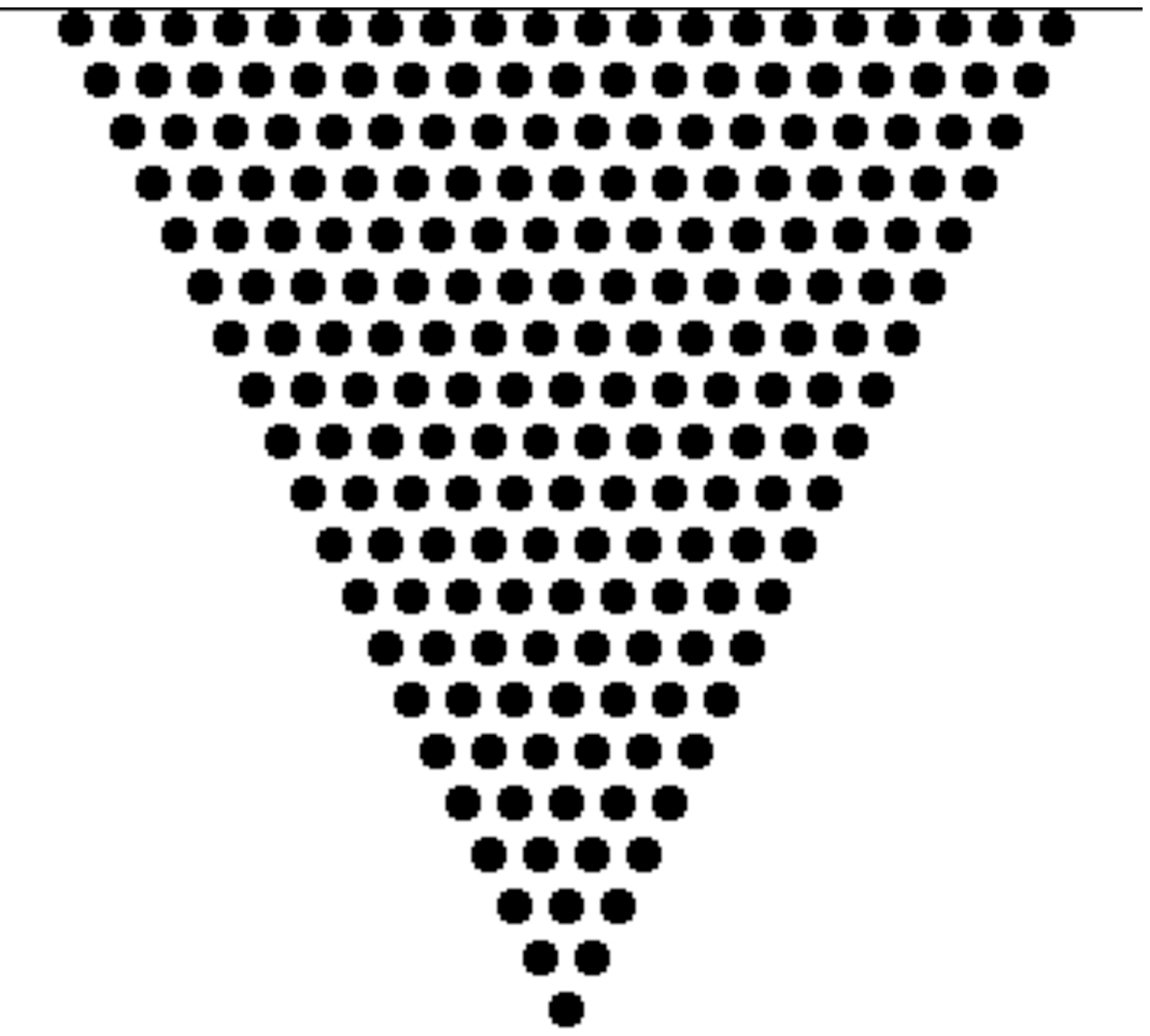}} \label{fig:init} \hspace{1cm}
\subfigure[]{ \includegraphics[scale=0.35]{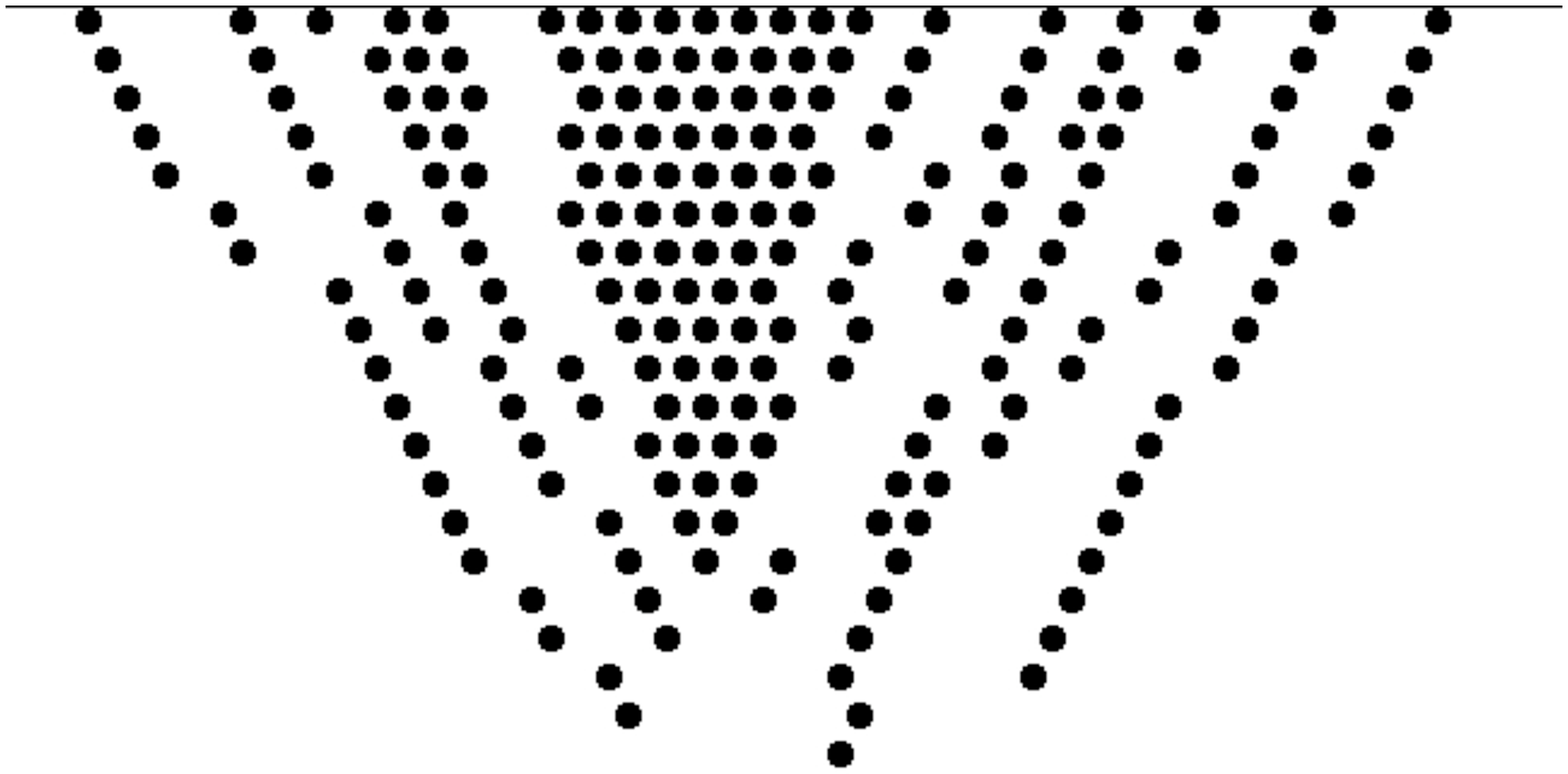}} \label{fig:pointconfig}
\end{center}
\caption{(a) The initial condition and (b) an example of a point configuration after some time.} \label{fig:particles}
\end{figure}

The evolution of the particles is as follows: each particle has two independent exponential clocks, a left and a right clock respectively. If the right (left) clock rings, the particle attempts  to jump to the right (left) by one. But in doing so it is forced to respect the interlacing condition according to the following two rules:  if the right (respectively left) clock of the particle at $x_{k}^m$ rings, then
\begin{enumerate}
\item if $x_k^m=x_{k}^{m-1}-1/2$  (or  $x_k^m=x_{k-1}^{m-1}+1/2$ in case the left clock rings) then it remains put.
 \item otherwise it jumps to the right by one and so do all particles $x_{k+l}^{m+l}$   with $x_{k+l}^{m+l}=x_{k}^m+l/2$ for $l=1,2,.\ldots$ (in case the left clock rings all  particles $x_{k}^{m+l}$ with   $x_{k}^{m+l}=x_k^m-l/2$ jump to the left for $l=1,2,\ldots$).
 \end{enumerate}
 Hence a particle that wants to jump is blocked by particles with  lower $m$-index, but it pushes particles at a higher $m$-index.

Next we specify the rate of the exponential clocks. For odd $m$ particles jump to the right with rate $\eps^{-1}>0$ and they jump to the left  with rate $\eps$. For even $m$ the  particles jump to the right with  rate $\eps$ and to the  left with rate $\eps^{-1}$. We are interested in the case where $\eps$ is small. Hence the particles for odd $m$  predominantly try to jump to the right. For even $m$ they mostly jump to  the left. However, they are still subject to the interlacing condition. Due to this condition, some particles that want to jump to the left (or right) are blocked, and even pushed to the right (or left) by particles at a lower level that want to travel right (or left).

If we let time evolve we will see mainly two clouds of particles travelling to the left and to the right respectively. See  Figure \ref{fig:pointconfig} for  typical point configuration after some time. For large  time we have the  macroscopic picture as shown in Figure \ref{fig:domainD}. With high probability the particles will be distributed in a domain $\mathcal D$ contained in the upper half plane. This domain consist of two parts $\mathcal D_1$ and $\mathcal D_2$. In $\mathcal D_2$ the particles are still densely pact as in the initial configuration (which means the particles did not have the chance to jump yet). In $\mathcal D_1$ the particles have a density strictly less than one. The boundary of $\mathcal D_1$ is a smooth curve except for two cusp points.  The cusps touch in the limit $\eps\downarrow 0$.

After rescaling the time parameter, the process has a well-defined limit for $\eps\downarrow 0$. In this limit the particles come in pairs. Indeed, by the interlacing condition some particles are blocked and even pushed by a particle at one level below that jumps in the reverse direction, thus forming a pair. The pairs are slanted to the right in the right half plane and slanted to the left in the left half plane, see also Figure \ref{fig:pointconfig}. The process for these pairs can be described in the following way: the process decouples in the sense  that we have two independent processes, one in the upper left quadrant and the other in the upper right quadrant. The process in the right quadrant is equivalent to the process where particles can only jump to the right. The process at the left is just its reflected version (hence the particle only jump to the left). This process is analyzed in \cite{BorFer} and from their results we recover that the limiting domain has a tacnode.


By standard arguments we can compute the limiting mean density of the particles in all cases. This settles the macroscopic behavior for the particles at  large time.  At the local scale we retrieve the well-known universality classes.

First consider $\eps>0$. If we zoom in at a point  away from the boundary, then we find that the local correlations are governed by one of the extensions to the discrete sine process that falls into the class as introduced in \cite{Bor}. If we zoom in at a point at the boundary, but not the cusps points, then we obtain the Airy process (see \cite{PS} and  \cite{F} for a review). The local correlations near the cusps are determined by the Pearcey process \cite{ABK,BH1,BH2,ORpearcey,TrW}. Since the proofs of these results follow from standard computations, they will be omitted. We do not need these statement for our main results.

In the case $\eps=0$ we have a decoupled system. In addition to a discrete sine process and the Airy process we also obtain the local correlations around the tacnode, which  are described by a process that is directly related to the GUE minor process (as we will prove).

The main  result of the paper is to derive the process around the tacnode in the  transition regime $\eps \downarrow 0$.  The main result is that we obtain a process that has not appeared in the literature (to the best of our knowledge). Naturally, it  interpolates between the Pearcey process and a process related to the GUE minor process.

In Section 2 we will state our main results and prove them in Section 3.
\section{Statement of results}

We start with some definitions. Let $\mathcal X$ be a discrete set. A point process on $\mathcal X$ is a probability measure on $2^\mathcal X$. A point process is completely determined by its correlation functions
\begin{equation}
\rho(X)=\mathop{{\rm Prob}} \{ Y \in 2^\mathcal X \mid X\subset Y\}.
\end{equation}
A point process is called determinantal, if there exists a kernel $K:\mathcal X\times \mathcal X\to \C$ such that
\begin{align}
\rho(X)=\det \left[K(x,y)\right] _{x,y\in X}.
\end{align}
For more details on determinantal point processes we refer to \cite{BorDet,HKPV,J,K,L,Sosh,Sosh2}. A determinantal point process is completely determined by its kernel.

Now let us return to the evolution on the interlacing particle system as decribed in the introduction. By stopping the process at time $t$, we get a random collection of points on the grid $\mathcal G$. Hence, the Markov process at time $t$ defines a point process on $\mathcal G$. In \cite{BorFer}, the authors proved that this is in fact a determinantal point process on $\mathcal  G$ with kernel $K$ given by
\begin{multline} \label{eq:kernel}
K(x_1,m_1;x_2,m_2)=
-\frac{\chi_{m_1<m_2}}{2\pi{\rm i}} \oint_{\Gamma_0} (1- \eps w)^{[m_1/2]-[m_2/2]}(1-  \eps/w)^{[(m_1+1)/2]-[(m_2+1)/2]}w^{[x_1]-[x_2]} \frac{{\rm d}w}{w}\\
+
\frac{1}{(2\pi{\rm i})^2} \oint_{\Gamma_0} {\rm d} w \oint_{\Gamma_{\eps,\eps^{-1}}}{\rm d}z\
\frac{
{\rm e}^{t(w+\frac{1}{w})}(1- \eps w)^{[m_1/2]}(  1-\eps/w)^{[(m_1+1)/2]}w^{[x_1]}
}
{{\rm e}^{t(z+\frac{1}{z})}(1- \eps z)^{[m_2/2]}( 1-\eps/z)^{[(m_2+1)/2]}z^{[x_2]}}
\frac{1}{z(w-z)},
\end{multline}
where $[x]$ is the largest integer less then $x$. Here $\Gamma_0$ is a contour that encircles the essential singularity $0$ but not the poles $\eps$ and $\eps^{-1}$. The contour $\Gamma_{\eps,\eps^{-1}}$  encircles $\eps$ and $\eps^{-1}$. Both $\Gamma_0$ and $\Gamma_{\eps,\eps^{-1}}$ have anti-clockwise orientation and do not intersect each other. Finally,
\begin{align}
\chi_{m_1<m_2}=\left\{\begin{array}{ll} 1, & \textrm{ if } m_1<m_2 \\ 0, & \textrm{otherwise.} \end{array}\right.
\end{align}

To be precise, the variables we use are different from \cite{BorFer}. We choose  a symmetric picture since we have particles jumping both  left and right, whereas the particular model that was analyzed in detail in \cite{BorFer} has particles jumping to the right only. Now \eqref{eq:kernel} is obtained by taking the kernel in \cite[Cor. 2.26]{BorFer} and substituting $y=[x]-[(m+1)/2)]$ and setting the  $\alpha_l$'s for even values of $l$ to $\eps$, and the other ones to $\eps^{-1}$. And finally, a conjugation by $(-\eps)^{[(m_2+1)/2]-[(m_1+1)/2)]}$ which does not effect the determinants in the correlation functions.\\

Our first result is that for large time we obtain the limiting situation as described in the Introduction and shown in Figure \ref{fig:domainD}.

\begin{figure}[t]\begin{center}
\subfigure[]{\begin{overpic}[scale=0.33]{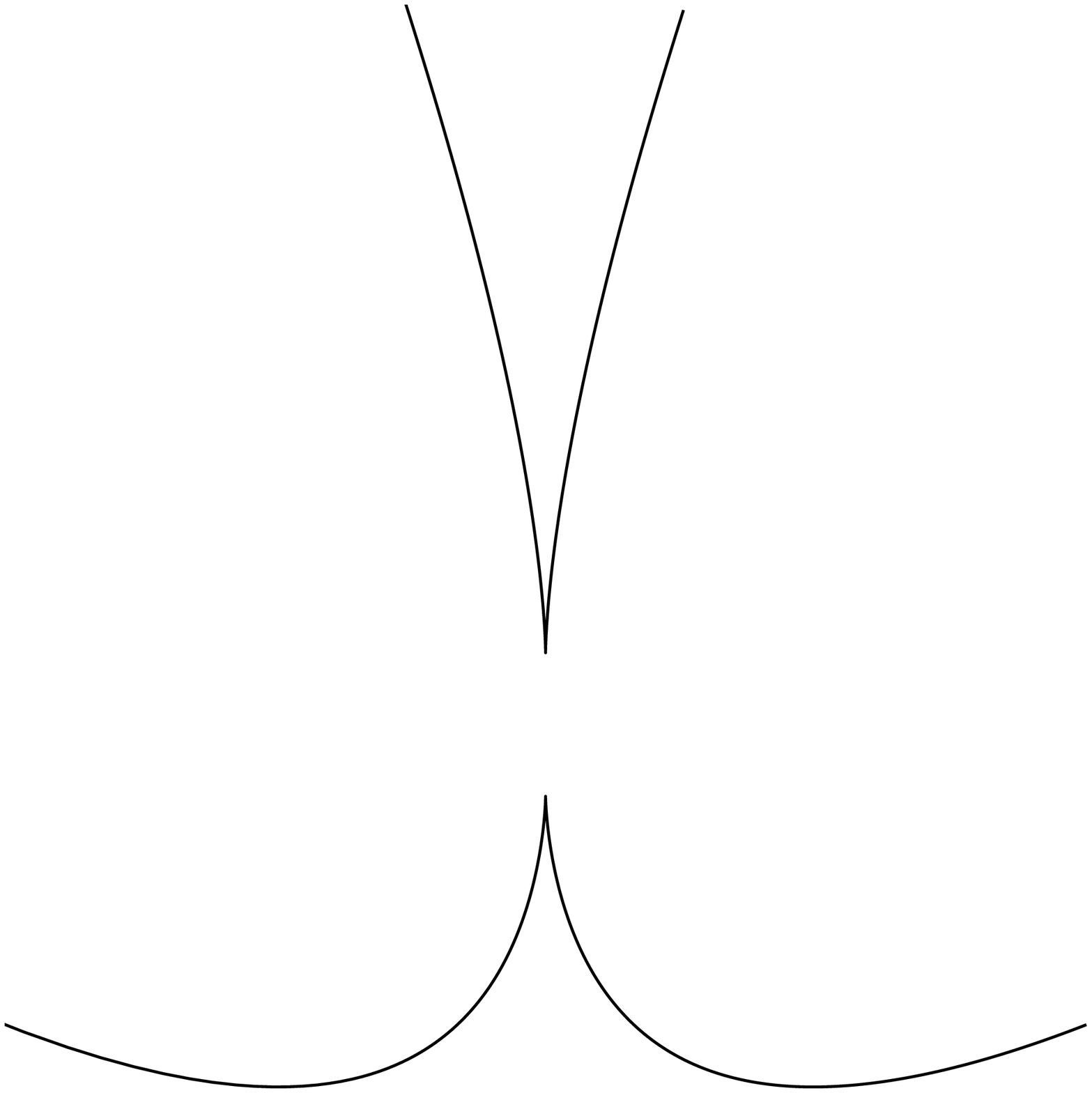}
\put(88.75,50.62){\makebox(0,0)[cc]{$\mathcal D_1$}}
\put(49.75,80.62){\makebox(0,0)[cc]{$\mathcal D_2$}}
\thicklines
\put(-1,1.5){\line(1,0){102}}
\put(90,0){\vector(1,0){10}}
\put(95,-3){\makebox(0,0)[cc]{$x$}}
\put(10,80){\vector(0,1){10}}
\put(5,85){\makebox(0,0)[cc]{$m$}}
\end{overpic}}
\hspace{1cm}
\subfigure[]{
\begin{overpic}[scale=0.33]{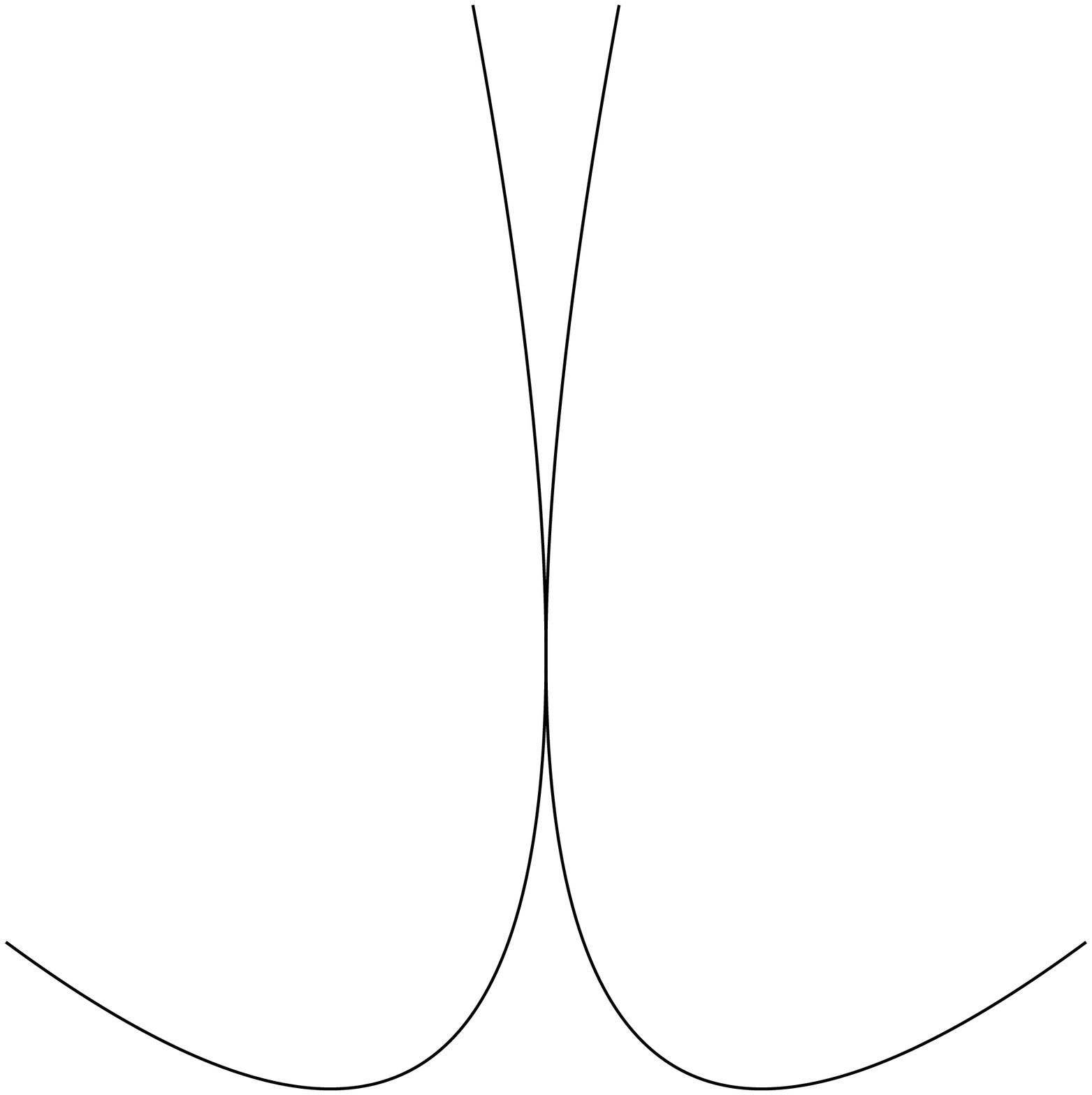}
\put(92.75,51.62){\makebox(0,0)[cc]{$\mathcal D_1$}}
\put(49.75,80.62){\makebox(0,0)[cc]{$\mathcal D_2$}}
\thicklines
\put(-1,1.5){\line(1,0){102}}
\put(90,0){\vector(1,0){10}}
\put(95,-3){\makebox(0,0)[cc]{$x$}}
\put(10,80){\vector(0,1){10}}
\put(5,85){\makebox(0,0)[cc]{$m$}}
\end{overpic}}

\caption{The typical shape of the limiting domains $\mathcal D_1$ and $\mathcal D_2$. The case $\eps>0$ at the left and $\eps=0$ at the right. In $\mathcal D_1$ the particle are still in the initial configuration. Outside $\mathcal D_1$ and $\mathcal D_2$ there are no particles.} \label{fig:domainD}

\end{center}
\end{figure}

\begin{theorem} \label{th:macro}
Let $\mathbb H=\{z\in \C \mid \Im x>0\}$ and   $F:\mathbb H \to \C$ given by
\begin{align}
F(z)=\tau (z+z^{-1})+\frac{\mu }{2} \log\left(1+\eps^{2}-\eps(z+z^{-1})\right)-\xi \log z.
\end{align}
Define $\mathcal D_1$  by
\begin{align}
\mathcal D_1=\{ (\xi,\mu)\in \R\times \R_+  \mid \exists z \in \mathbb H  \quad  F'(z)=0\}.
\end{align}
The boundary $\partial D_1$ has two cusp points located at
\begin{align}
(0,\eps+\eps^{-1}\pm 2).
\end{align}
Set
\begin{align}
\left\{
\begin{array}{l}
t=\tau L\\
m=[\mu L]\\
x=[\xi L]
\end{array}
\right.\end{align}
then the limiting mean density is given by
\begin{align}
\lim_{L\to \infty} K(x, m, x,m)=
 \left\{
 \begin{array}{ll}
 1, & (\xi,\mu)\in \mathcal D_2,\\
 \frac{1}{\pi} \arg z(\xi,\mu) , & (\xi,\mu)\in \mathcal D_1,\\
 0, & \textrm{ otherwise}.
 \end{array}
 \right.
\end{align}
where $z(\xi,\mu)$ is the unique solution in the upper half plane of the equation $F'(z)=0$.
\end{theorem}

The proof of this result follows from standard steepest descent analysis on the double integral formula for the kernel. Because the proof is standard, we will omit it in this paper. We do not use Theorem \ref{th:macro} in the sequel. See \cite{BorFer} for a proof of a similar statement in a comparable situation or \cite{O} for an exposition of the steepest descent technique on double integral formulas. \\

\begin{remark}
The boundary $\partial \mathcal D_1$ can be explicitly computed. Indeed, it is clear that
\begin{align}
\partial \mathcal D_1=\{(\xi,\mu)\in \R\times \R_+ \mid \exists z \in \R  \quad F'(z)=0 \wedge  F''(z)=0\}.
\end{align}
Now for each $z\in \R$ we have that
\begin{align}
\left\{
\begin{array}{l}
F'(z)=0\\
F''(z)=0
\end{array}\right.
\end{align}
is system of equations that is linear in $\xi$ and $\mu$. So we can easily express $(\xi,\mu)$ as a function of $z\in \R$. In fact, it is easily checked that for each $z\in \R \setminus \{0,\eps,\eps^{-1}\}$ the corresponding $(\xi,\mu)$ satisfy $\xi\in\R$ and $\mu\geq 0$ so that $(\xi,\mu)$ is a point on $\partial \mathcal D_1$. Therefore, the closure of the image of $\R\setminus\{0,\eps,\eps^{-1}\}$ under this map $z\mapsto (\xi,\mu)$ gives the boundary. The cusp pointing down corresponds to $z=-1$ and the cusp pointing up to $z=1$. The boundary touches the $x$-axis at $z=\eps,\eps^{-1}$.
\end{remark}
From Theorem \ref{th:macro} it follows that we can achieve the situation of a tacnode in the following way. For every fixed $\eps$ the cusp points on the boundary differ by $4$. The location tends to infinity when we take the limit $\eps\downarrow 0$. By  rescaling $\mu$ with $\eps$ the cusp points have  a limit  as $\eps\downarrow 0$ and the gap between the cusp points  vanishes, resulting in a tacnode.

 From Theorem \ref{th:macro} and \eqref{eq:kernel} we expect to arrive at a process around the tacnode when we scale \begin{align}\label{eq:newparamCrK}
\left\{
\begin{array}{l}
t=\epsilon L\\
m_j=[L^2(1+\mu_j /L)]\\
\eps=\epsilon/L
\end{array}
\right.
\end{align}
However, it is less clear how to describe the process that arises at the cusp. As shown in Figure \ref{fig:demo}, there will be long vertical strings of particles. In fact, the length of these strings will be of order $L$ and hence the density of the particles  at this scale will diverge. Therefore we do not obtain a point process in the limit if we consider the process on the particles.

\begin{figure}
\begin{center}
\includegraphics[scale=0.5]{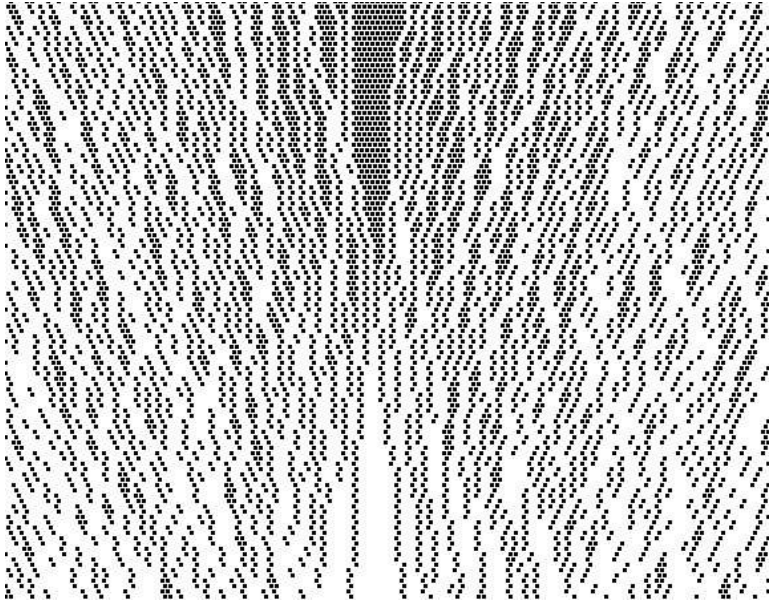}
\end{center}
\caption{A  close up picture of the process, for $\eps=0$, around the point where the cusps touch.  Note that the particle come in pairs that are slanted to the right in the right half and to the left in the left half. Around the tacnode, there are forming long vertical strings of particles. }\label{fig:demo}
\end{figure}
There are several ways of constructing a meaningful process in the limit. Perhaps the most straightforward approach is the following: instead of allowing $\mu_j$  to be a free variable, we choose $N\in \N$ and fix $\mu_j \in \R$ for $j=1,\ldots, N$ with $\mu_j\neq \mu_k$ if $j \neq k$ and cut  the process at the section $m_j=[L^2(1+\mu_j/L)]$. For simplicity, we will also shift the particles on any even $m_j$-section to the left by a half. In this way,  we obtain a determinantal point process on $\Z\times (1,\ldots,N)$ for each $L$ with  correlation functions
\begin{align}
  \mathop{\mathrm{Prob}}( \{\textrm{particle at }(x_k,j_k)\in \Z \times (1,\ldots, N) \mid k=1,\ldots,n\})=\det \left(K(x_k,m_{j_k},x_l,m_{j_l})\right)_{k,l=1}^n
\end{align}
for all $n\in \N$.

Now to obtain the limiting process as $L \to \infty$ for the limiting process on $\Z\times (1,\ldots,N)$, it suffices to compute the pointwise limit for the kernel $K$. The limit is given in the next theorem, which is the main result of the present paper.

\begin{figure}[t]\begin{center}
\input{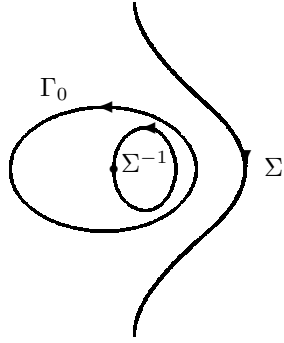}
\end{center}
\caption{The contours of integration in the kernel $\mathcal K^\epsilon$ } \label{fig:contoursCrK}

\end{figure}
\begin{theorem} \label{th:Keps}
With $t,m_j$ as in \eqref{eq:newparamCrK}  we have that
\begin{align}
\lim_{L\to \infty} K(x_1,m_1,x_2,m_2)=\mathcal K^{\epsilon}(x_1,\mu_1,x_2,\mu_2)
\end{align}
for $(x_j,\mu_j)\in \Z\times \R$, where
\begin{multline} \label{eq:kernelCr}
\mathcal K^{\epsilon}(x_1,\mu_1,x_2,\mu_2) =
-\frac{1}{\chi_{\mu_1<\mu_2}} \int_{\Gamma_0} {\rm e}^{\epsilon(\mu_2-\mu_1)(w+1/w)} w^{x_1-x_2-1}{\rm d} w
\\
+\frac{1}{(2\pi{\rm i})^2}\oint_{\Gamma_0} \oint_{\Sigma\cup \Sigma^{-1}}\frac{{\rm e}^{\epsilon \mu_2 (z+\frac{1}{z})+\frac{\epsilon^2}{2}(z+\frac{1}{z})^2}w^{x_1}}
{{\rm e}^{\epsilon \mu_1 (w+\frac{1}{w})+\frac{\epsilon^2}{2}(w+\frac{1}{w})^2}z^{x_2}}
\frac{{\rm d}z {\rm d}w}{z(w-z)}
\end{multline}
The contours of integration and their orientation are as indicated in Figure \ref{fig:contoursCrK}. More precisely, $\Gamma_0$ is a contour encircling the origin with counter clockwise orientation. The contour $\Sigma$ is a contour connecting $-{\rm i}\infty$ to ${\rm i}\infty$ that does not intersect $\Gamma_0$ and stays in the right-half plane.
\end{theorem}
To the best of our knowledge, the kernel $K^\epsilon$  has not appeared in the existing literature yet.

It is not difficult to show that after inserting the new parameters given in \eqref{eq:newparamCrK} in the integrands in \eqref{eq:kernel} and taking the pointwise limit as $L\to \infty$,  one obtains the integrands as given in the right-hand side of \eqref{eq:kernelCr}. However,  there is an important technical issue that needs to be taken care of. Note that the integrand in the double integral contains poles at $\eps$ and $\eps^{-1}$, and also an essential singularity at $0$. By taking the limit, the pole approaches  the essential singularity at the origin, which complicates the contour deformation in the analysis. \\

A different way of creating a point process is the following. Instead of considering the location of the particles, one could consider the statistics of the upper endpoints of the vertical strings of particles. We will restrict our process only to the odd $m$-sections. Then $(x,m)$ is defined to be a upper endpoint if there is a particle at $(x,m)$ but there is no particle at $(x,m+2)$. The upper endpoints form a point process with correlation functions
 \begin{multline}
 \tilde \rho_N((x_1,m_1),\ldots,(x_N,m_N))=\mathop{\mathrm{Prob}}\left(\textrm{particle at } (x_j,m_j) \textrm{ and no particle at } (x_j,m_j+2) \mid \, j=1,\ldots,N\right)
 \end{multline}
 for $(x_j,m_j)\in \Z\times (2\N-1)$ and $N\in \N$.

 \begin{theorem}\label{th:endpoints}
 With $t=\epsilon L$ and $m_j$ the closest odd integer to $L^2(1+\mu_j/L)$, we have that
 \begin{align}\label{eq:corendpoints}
\lim_{L\to \infty} L^N \tilde \rho_N((x_1,m_1),\ldots,(x_N,m_N))=\epsilon^N \det\left(\mathcal K^\epsilon(x_i-1,\mu_i,x_j,\mu_j)+K^\eps(x_i+1,\mu_i,x_j,\mu_j) \right)_{i,j}^N
\end{align}
for $(x_j,\mu_j)\in \Z\times\R$ for $j=1,\ldots,N$.
\end{theorem}

\begin{remark}
Another  argument for the fact that the lengths of the vertical strings of particles must be of order $L$, is that one can prove that the process restricted to two horizontal sections that are close, is just constructed of two copies of the process restricted to one of these horizontal sections. To be precise, let $\mu_1\in \R$ and take $m_1=[L^2(1+\mu_1/L)]$. For each $L$ let $m_2$ be such that $m_1-m_2=o(L)$ as $L\to \infty$. For simplicity, we assume that $m_2 \geq m_1$. Then it is not difficult to prove that
\begin{align}
\lim_{L \to \infty} K(x_1,m_1,x_2,m_1)&=\mathcal K^\epsilon(x_1,\mu_1,x_2,\mu_1),\\
\lim_{L \to \infty} K(x_1,m_2,x_2,m_2)&=\mathcal K^\epsilon(x_1,\mu_1,x_2,\mu_1),\\
\lim_{L \to \infty} K(x_1,m_2,x_2,m_1)&=\mathcal K^\epsilon(x_1,\mu_1,x_2,\mu_1),\\
\lim_{L \to \infty} K(x_1,m_1,x_2,m_2)&=-\delta_{x_1,x_2}+ \mathcal K^\epsilon(x_1,\mu_1,x_2,\mu_1),
\end{align}
for all $x_1,x_2\in \Z$. Hence,  for $x_1,\ldots,x_l, y_1,\dots,y_k$  we have
\begin{multline}
\lim_{L\to \infty} \mathrm{Prob}(\textrm{particles at }(x_1,m_1),\ldots,(x_l,m_1),(y_{1},m_2),\ldots, (y_k,m_2))\\
=\det
\begin{pmatrix}
\left(\mathcal K^\epsilon(x_i,\mu_1,x_j,\mu_1)\right)_{i,j=1}^l & \left(\mathcal K^\epsilon(x_i,\mu_1,y_j,\mu_1)-\delta_{x_i.x_j}\right)_{i=1,j=1}^{l,k} \\
\left(\mathcal K^\epsilon(y_i,\mu_1,x_j,\mu_1)\right)_{i=1,j=1}^{k,l} & \left(\mathcal K^\epsilon(y_i,\mu_1,y_j,\mu_1)\right)_{i,j=1}^k.
\end{pmatrix}
\end{multline}
This implies that (in the limit $L\to \infty$) the process on the line  $m_2$ is just a copy of the process on the line $m_1$. \end{remark}

We will now derive some properties of the kernel $\mathcal K^{\eps}$. The following proposition shows the symmetry in the kernel.

\begin{proposition} \label{prop:sym}
We have that
\begin{enumerate}
\item $\mathcal K^\epsilon(-x_1,\mu_1,-x_2,\mu_2)=\mathcal K^\epsilon(x_1-1,\mu_1,x_2-1,\mu_2)$.
\item $(-1)^{x_1-x_2}\mathcal K^\epsilon(x_1,-\mu_1,x_2,-\mu_2)=\delta_{(x_1,\mu_1),(x_2,\mu_2)}- \mathcal K^\epsilon(x_1,\mu_1,x_2,\mu_2)$
\end{enumerate}
for all $(x_j,\mu_j)\in \Z\times \R$.
\end{proposition}
The first property shows that our point process is invariant with respect to the transform $x\mapsto -1-x$.

To interpret the second symmetry property in this proposition, we note that if $\mathcal P$ is a determinantal point process on a discrete set $\mathcal X$ with kernel $K$, then we have that $1-K$ is the kernel of the determinantal point process $\mathcal P'$ defined by $\mathcal P'(X)=\mathcal P(\mathcal X\setminus X)$, for $X\subset \mathcal X$. This process is sometimes referred to as the complementary process. It is obtained by replacing particles with holes and vice versa. For more details on this particle hole involution  we refer to the appendix of \cite{BOO}.\\

In the final results of this paper we investigate the limiting behavior of the kernel as $\epsilon\downarrow 0$ and $\epsilon \to\infty$. We start with the first case.
\begin{figure}[t]
\begin{center}
\subfigure{\ifx\JPicScale\undefined\def\JPicScale{0.45}\fi
\unitlength \JPicScale mm

\begin{picture}(80,110)(-25,0)

\linethickness{0.4mm}
\thicklines
\put(-20,0){\vector(1,0){100}}
\put(-20,0){\vector(0,1){110}}

\put(-28,100){\makebox(0,0)[cc]{$\mu$}}
\put(80,-5){\makebox(0,0)[cc]{$x$}}

\put(0,-5){\makebox(0,0)[cc]{\small{$0$}}}
\put(20,-5){\makebox(0,0)[cc]{\small{$1$}}}
\put(40,-5){\makebox(0,0)[cc]{\small{$2$}}}
\put(60,-5){\makebox(0,0)[cc]{\small{$3$}}}


\put(0,0){\line(0,1){2.5}}
\put(20,0){\line(0,1){2.5}}
\put(40,0){\line(0,1){2.5}}
\put(60,0){\line(0,1){2.5}}

\put(0,50){\circle{5}}

\put(0,52.5){\line(0,1){57.5}}

\put(20,75){\circle{5}}
\put(20,25){\circle{5}}

\put(20,77.5){\line(0,1){32.5}}
\put(20,27.5){\line(0,1){22.5}}

\put(40,90){\circle{5}}
\put(40,60){\circle{5}}
\put(40,15){\circle{5}}

\put(40,92.5){\line(0,1){17.5}}
\put(40,62.5){\line(0,1){12.5}}
\put(40,17.5){\line(0,1){7.5}}

\put(60,95){\circle{5}}
\put(60,65){\circle{5}}
\put(60,35){\circle{5}}
\put(60,10){\circle{5}}

\put(60,97.5){\line(0,1){12.5}}
\put(60,67.5){\line(0,1){22.5}}
\put(60,37.5){\line(0,1){22.5}}
\put(60,12.5){\line(0,1){2.5}}

\put(2.5,50){\psline[linestyle=dotted](0,0)(0.7,0)}

\put(22.5,75){\psline[linestyle=dotted](0,0)(0.7,0)}
\put(22.5,25){\psline[linestyle=dotted](0,0)(0.7,0)}

\put(42.5,90){\psline[linestyle=dotted](0,0)(0.7,0)}
\put(42.5,60){\psline[linestyle=dotted](0,0)(0.7,0)}
\put(42.5,15){\psline[linestyle=dotted](0,0)(0.7,0)}

\put(62.5,95){\psline[linestyle=dotted](0,0)(0.7,0)}
\put(62.5,65){\psline[linestyle=dotted](0,0)(0.7,0)}
\put(62.5,35){\psline[linestyle=dotted](0,0)(0.7,0)}
\put(62.5,10){\psline[linestyle=dotted](0,0)(0.7,0)}

\end{picture}}
\hspace*{2cm}
\subfigure{\ifx\JPicScale\undefined\def\JPicScale{0.45}\fi
\unitlength \JPicScale mm

\begin{picture}(80,110)(-25,0)

\linethickness{0.4mm}
\thicklines

\put(-20,0){\vector(1,0){100}}
\put(-20,0){\vector(0,1){110}}


\put(0,-5){\makebox(0,0)[cc]{\small{$0$}}}
\put(20,-5){\makebox(0,0)[cc]{\small{$1$}}}
\put(40,-5){\makebox(0,0)[cc]{\small{$2$}}}
\put(60,-5){\makebox(0,0)[cc]{\small{$3$}}}

\put(-28,30){\makebox(0,0)[cc]{\small{$\mu_1$}}}
\put(-28,40){\makebox(0,0)[cc]{\small{$\mu_2$}}}
\put(-28,70){\makebox(0,0)[cc]{\small{$\mu_3$}}}
\put(-28,100){\makebox(0,0)[cc]{\small{$\mu_4$}}}


\put(0,0){\line(0,1){2.5}}
\put(20,0){\line(0,1){2.5}}
\put(40,0){\line(0,1){2.5}}
\put(60,0){\line(0,1){2.5}}


\put(0,50){\circle{5}}

\put(0,52.5){\line(0,1){57.5}}

\put(20,75){\circle{5}}
\put(20,25){\circle{5}}

\put(20,77.5){\line(0,1){32.5}}
\put(20,27.5){\line(0,1){22.5}}

\put(40,90){\circle{5}}
\put(40,60){\circle{5}}
\put(40,15){\circle{5}}

\put(40,92.5){\line(0,1){17.5}}
\put(40,62.5){\line(0,1){12.5}}
\put(40,17.5){\line(0,1){7.5}}

\put(60,95){\circle{5}}
\put(60,65){\circle{5}}
\put(60,35){\circle{5}}
\put(60,10){\circle{5}}

\put(60,97.5){\line(0,1){12.5}}
\put(60,67.5){\line(0,1){22.5}}
\put(60,37.5){\line(0,1){22.5}}
\put(60,12.5){\line(0,1){2.5}}

\put(-20,100){\psline[linestyle=dashed](4.2,0)(0,0)}
\put(-20,70){\psline[linestyle=dashed](4.2,0)(0,0)}
\put(-20,40){\psline[linestyle=dashed](4.2,0)(0,0)}
\put(-20,30){\psline[linestyle=dashed](4.2,0)(0,0)}

\put(0,100){\circle*{5}}
\put(20,100){\circle*{5}}
\put(40,100){\circle*{5}}
\put(60,100){\circle*{5}}

\put(0,70){\circle*{5}}
\put(40,70){\circle*{5}}
\put(60,70){\circle*{5}}

\put(20,40){\circle*{5}}
\put(60,40){\circle*{5}}
\put(20,30){\circle*{5}}

\end{picture}}
\end{center}
\caption{From the GUE minor process to the point process with kernel given in \eqref{eq:kerneleps=0}. The open circles represent the point $(x,y_x^l)$ from $l=0,\ldots,x$ and $x=0,1,\ldots$. In the left picture we draw the vertical lines. The dotted lines are only auxillary. In the right picture, we draw the lines associated to the choice of the $\mu_j$. The solid circles are the intersection point of the dashed horizontal and solid vertical line. The solid circles describe the process with kernel \eqref{eq:kerneleps=0}.}
\label{fig:GUE}
\end{figure}
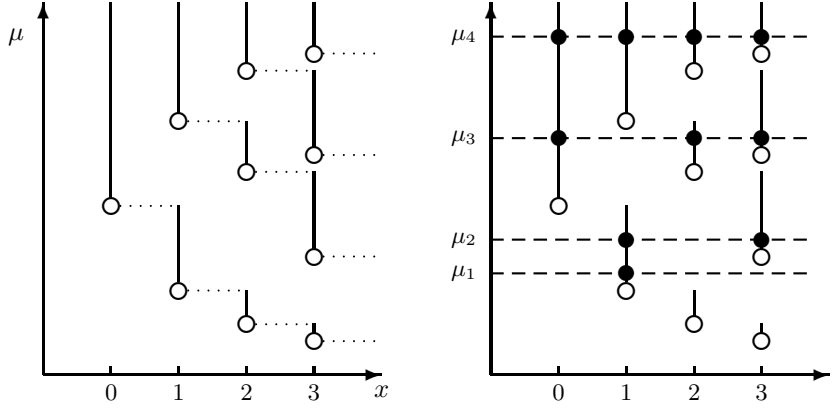
\begin{theorem}\label{th:hermite}
For $x_1,x_2<0$ we have that
\begin{multline}\label{eq:kerneleps=0a}
\lim_{\epsilon \downarrow 0}\epsilon^{x_1-x_2} \mathcal K^{\epsilon}(x_1, \mu_1,x_2,\mu_2)= -\chi_{\mu_1<\mu_2} \chi_{x_1\leq x_2} (\mu_2-\mu_1)^{x_2-x_1}\\
+\frac{1}{(2\pi{\rm i})^2} \oint_{\Gamma_0} \int_\Sigma \frac{{\rm e}^{\mu_2 z+\frac{1}{2}z^2}}{{\rm e}^{\mu_1 w+\frac{1}{2} w^2}}\frac{w^{x_1}}{z^{x_2+1}} \frac{{\rm d}z{\rm d}w}{w-z},
\end{multline}
 and for $x_1,x_2\geq 0$.
\begin{multline}\label{eq:kerneleps=0}
\lim_{\epsilon \downarrow 0}\epsilon^{x_2-x_1} \mathcal K^{\epsilon}(x_1, \mu_1,x_2,\mu_2)= -\chi_{\mu_1<\mu_2} \chi_{x_2\leq x_1} (\mu_2-\mu_1)^{x_2-x_1}\\
+\frac{1}{(2\pi{\rm i})^2} \oint_{\Gamma_0} \int_\Sigma \frac{{\rm e}^{\mu_2 z+\frac{1}{2}z^2}}{{\rm e}^{\mu_1 w+\frac{1}{2} w^2}}\frac{z^{x_2}}{w^{x_1+1}} \frac{{\rm d}z{\rm d}w}{w-z}.
\end{multline}
 Moreover, the limit of the  kernel vanishes if neither of the two conditions on $x_1$ and $x_2$ are satisfied.
 \end{theorem}
The one dimensional version $\mu_1=\mu_2$ of this kernel has appeared before in the literature, see \cite{BO}.  By expanding the term $(w-z)^{-1}$ we can express the double integral as a sum of products of Hermite polynomials.

By inserting \eqref{eq:kerneleps=0} (or \eqref{eq:kerneleps=0a}) in \eqref{eq:corendpoints} we find the limiting process for the endpoints in the right (or left) half plane. It turns out that in each of the half planes one of the kernels in the determinant tends to zero as $\epsilon \downarrow 0$. In fact, the limiting kernel is the kernel corresponding to the GUE minor process, see for example \cite{JN,ORgue}. Indeed, in  the case $x_1,x_2\geq 0$ we rewrite \eqref{eq:corendpoints} as
\begin{align}
\lim_{L\to \infty} L^N \tilde \rho_N((x_1,m_1),\ldots,(x_N,m_N))= \det\left( \epsilon^{x_j+1-x_i} \mathcal K^\epsilon(x_i-1,\mu_i,x_j,\mu_j)+ \eps^2\epsilon^{x_j-1-x_j}  K^\epsilon(x_i+1,\mu_i,x_j,\mu_j) \right)_{i,j}^N,
\end{align}
and then by \eqref{eq:kerneleps=0} we find
\begin{align}
\lim_{\epsilon \downarrow 0} \lim_{L\to \infty} L^N \tilde \rho_N((x_1,m_1),\ldots,(x_N,m_N))= \det\left(K_{\mathrm{GUE}}(x_i,\mu_i,x_j,\mu_j) \right)_{i,j}^N,
\end{align}
where
\begin{multline}
K_{\mathrm{GUE}}(x_i,\mu_i,x_j,\mu_j) = -\chi_{\mu_1<\mu_2} \chi_{x_2< x_1} (\mu_2-\mu_1)^{x_2-x_1-1}
+\frac{1}{(2\pi{\rm i})^2} \oint_{\Gamma_0} \int_\Sigma \frac{{\rm e}^{\mu_2 z+\frac{1}{2}z^2}}{{\rm e}^{\mu_1 w+\frac{1}{2} w^2}}\frac{z^{x_2}}{w^{x_1}} \frac{{\rm d}z{\rm d}w}{w-z}.
\end{multline}
The case $x_1,x_2<0$ is similar. By comparing with \cite[Def. 1.2]{JN} we see that the kernel $K_{GUE}$ describes the GUE minor process.
\begin{remark}
Based upon the relation between the construction of the two processes just above and below Theorem \ref{th:Keps}, we see how the process defined by the kernel at the right-hand side of \eqref{eq:kerneleps=0} can be constructed out of  the GUE minor process explicitly: pick a point configuration random from the GUE minor process and denote the points by $(x,y^l_x)\in  \N\times \R$. At each vertical $x$ section we draw $x$ vertical line segments. Each segment has $y_x^l$ as a lower endpoint and $y_{x-1}^{l-1}$ as the upper endpoint. Here we define $y_x^0=\infty$. See also Figure \ref{fig:GUE}.
We next fix $N$ different real numbers $\mu_j\in \R$. Then we define a point process in $\N\times \{1,\ldots,N\}$ by
\begin{align}
\mathrm{Prob} \{\textrm{particle at } (x_i,j_i), \ i=1,\ldots k)\}=\mathrm{Prob}\{\textrm{Each }(x_i,\mu_{j_i}) \textrm{ is on a line segment}\}
\end{align}
The conclusion is that the new process is in fact a determinantal point process with kernel as in \eqref{eq:kerneleps=0}. \\
\end{remark}

The second situation that can be obtained is that for $\epsilon \to \infty$, in which the cusps should separate. Hence we expect to obtain the Pearcey process when we take the simultaneous limit $\mu_j\to -\infty$ (or $\mu_j\to +\infty$). The following Theorem states that we indeed find the Pearcey process at the lowest of the two cusps.

\begin{theorem}\label{th:pearcey}
Set
\begin{align}\label{eq:newparamP}
\left\{
\begin{array}{l}
\epsilon=M\\
\mu_j=- M(1-\nu_j/2M)\\
x=[\xi_j M^{1/2}]
\end{array}
\right.
\end{align}
Then
\begin{multline}\label{eq:kernelP}
\lim_{M\to\infty} \frac{{\rm e}^{M (\nu_1-\nu_2)}}{M^{1/2}}\mathcal K^{\eps}(x_1,\mu_1,x_2,\mu_2)=
-\frac{\chi_{\nu_1<\nu_2}}{2\pi{\rm i}}\int_{-{\rm i}\infty}^{{\rm i}\infty} {\rm d}w\ {\rm e}^{(\nu_2-\nu_1)w^2-(\xi_2-\xi_1)w}\\
+\frac{1}{(2 \pi{\rm i})^2}\int_{-{\rm i}\infty}^{{\rm i}\infty} \int_{\mathcal C} \frac{{\rm e}^{\frac{1}{2}z^4+\nu_2 z^2-\xi_2 z}}{ {\rm e}^{\frac{1}{2}w^4+\nu_1 w^2-\xi_1 w}}\frac{{\rm d}w {\rm d}z}{w-z}
\end{multline}
The contour $\mathcal C$ conists of four rays, from $\pm {\rm e}^{\pi {\rm i}/4}\infty $ to $0$ and from $0$ to   $\pm {\rm e}^{3\pi {\rm i}/4}\infty $ (see also  Figure \ref{fig:pearcey}).
\end{theorem}
The kernel given \eqref{eq:kernelP} is known in the literature as the extended Pearcey kernel that describes the Pearcey process, see  \cite{ABK,BH1,BH2,ORpearcey,TrW} for more details.

To conclude this section, note that if we combine Theorem \ref{th:pearcey} with the second symmetry property in Proposition \ref{prop:sym} then we see that the complementary process  near the top cusp  locally converges to the Pearcey process, as expected.

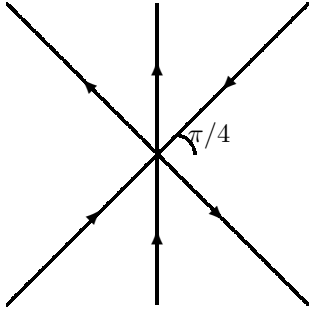
\begin{figure}[t] \begin{center}
\ifx\JPicScale\undefined\def\JPicScale{0.5}\fi
\unitlength \JPicScale mm
\begin{picture}(90,90)(0,0)
\thicklines
\linethickness{0.42mm}
\multiput(10,10)(0.12,0.12){667}{\line(1,0){0.12}}
\linethickness{0.42mm}
\multiput(10,90)(0.12,-0.12){667}{\line(1,0){0.12}}
\linethickness{0.42mm}
\put(50,10){\line(0,1){80}}
\linethickness{0.3mm}
\put(50,65){\line(0,1){10}}
\put(50,75){\vector(0,1){0.12}}
\linethickness{0.3mm}
\multiput(70,70)(0.12,0.12){31}{\line(1,0){0.12}}
\put(67,67){\vector(-1,-1){0.12}}
\linethickness{0.3mm}
\multiput(30,70)(0.12,-0.12){63}{\line(1,0){0.12}}
\put(30,70){\vector(-1,1){0.12}}
\linethickness{0.3mm}
\multiput(62.5,37.5)(0.12,-0.12){26}{\line(1,0){0.12}}
\put(67.62,32.38){\vector(1,-1){0.12}}
\linethickness{0.3mm}
\multiput(30,30)(0.12,0.12){42}{\line(1,0){0.12}}
\put(35,35){\vector(1,1){0.12}}
\linethickness{0.3mm}
\put(50,25){\line(0,1){5}}
\put(50,30){\vector(0,1){0.12}}
\put(80,55){\makebox(0,0)[cc]{}}

\put(60,55){\makebox(0,0)[cc]{}}

\put(60,55){\makebox(0,0)[cc]{}}

\linethickness{0.3mm}
\multiput(59.94,50.51)(0.06,-0.51){1}{\line(0,-1){0.51}}
\multiput(59.84,51.02)(0.11,-0.5){1}{\line(0,-1){0.5}}
\multiput(59.69,51.51)(0.15,-0.49){1}{\line(0,-1){0.49}}
\multiput(59.49,51.98)(0.1,-0.24){2}{\line(0,-1){0.24}}
\multiput(59.24,52.44)(0.12,-0.23){2}{\line(0,-1){0.23}}
\multiput(58.95,52.86)(0.14,-0.21){2}{\line(0,-1){0.21}}
\multiput(58.63,53.26)(0.11,-0.13){3}{\line(0,-1){0.13}}
\multiput(58.26,53.63)(0.12,-0.12){3}{\line(1,0){0.12}}
\multiput(57.86,53.95)(0.13,-0.11){3}{\line(1,0){0.13}}
\multiput(57.44,54.24)(0.21,-0.14){2}{\line(1,0){0.21}}
\multiput(56.98,54.49)(0.23,-0.12){2}{\line(1,0){0.23}}
\multiput(56.51,54.69)(0.24,-0.1){2}{\line(1,0){0.24}}
\multiput(56.02,54.84)(0.49,-0.15){1}{\line(1,0){0.49}}
\multiput(55.51,54.94)(0.5,-0.11){1}{\line(1,0){0.5}}
\multiput(55,55)(0.51,-0.06){1}{\line(1,0){0.51}}

\put(63.75,55){\makebox(0,0)[cc]{$\pi/4$}}

\end{picture}
\caption{The contours of integration for the Pearcey kernel.}\label{fig:pearcey}
\end{center}
\end{figure}

\section{Proofs}

In this section we prove our results.

\subsection{Proof of Theorem \ref{th:Keps}}

Let us first introduce some notation. Write the kernel in \eqref{eq:kernel} as
 \begin{multline}
K(x_1, m_1,x_2,m_2) = -\frac{\chi_{m_1<m_2}}{2\pi{\rm i}} \oint_{\Gamma_0} \frac{G_{t,m_1,x_1}(z)}{ G_{t,m_2,x_2}(z)} \frac{{\rm d}z}{z}+\frac{1}{(2\pi{\rm i})^2} \oint_{\Gamma_0}\oint_{\Gamma_{\eps,\eps^{-1}}} \frac{G_{t,m_1,x_1}(w)}{G_{t,m_2,x_2}(z)}\frac{1}{z(w-z)} \ {\rm d}z {\rm d}w,
\end{multline}
where
\begin{align} \label{eq:G}
G_{t,m,x}(w)={\rm e}^{t(w+1/w)}(1-\eps w)^{[m_1/2]}(1-\eps/w)^{[(m_1+1)/2]}w^{[x]}.
\end{align}
Let us for a moment ignore the contours of integration and take the limit $L\to \infty$ for the integrand.  It is not difficult to see that we have the following pointwise limit
  \begin{align}\label{eq:limitG}
 \lim_{L\to \infty}  \frac{G_{t,m_1,x_1}(w)}{G_{t,m_2,x_2}(z)}=\frac{{\rm e}^{ \varepsilon \mu_2 (z+1/z)+\eps^2 (z+1/z)^2}}{{\rm e}^{ \varepsilon \mu_2 (w+1/w)+\eps^2 (w+1/w)^2}} \frac{w^{[x_1]}}{z^{[x_2]}}.
 \end{align}
 for $z,w\in \C\setminus \{0\}$. Hence the difficulty in the proof lies in the fact that we have to control the contours of integration (which clearly depend on $\eps$ and hence $L$) while taking the limit.

 As a first step we prove the following result.

\begin{lemma} \label{lem:deform}
With $G$ as in \eqref{eq:G} and $K$ as in \eqref{eq:kernel} we have that
\begin{multline}
K(x_1,m_1,x_2,m_2)= -\frac{\chi_{m_1<m_2}}{2\pi{\rm i}} \oint_{\Gamma_{0,\eps}}{\rm d} z \ \frac{G_{t,m_1,x_1}(z)}{z G_{t,m_2,x_2}(z)} +\frac{1}{(2\pi{\rm i})^2} \oint_{\Gamma_0,\eps}  \oint_{\Gamma_\eps\cup\Gamma_{\eps^{-1}}}\frac{G_{t,m_1,x_1}(w)}{ G_{t,m_2,x_2}(z)} \frac{1}{z(w-z)}\end{multline}
Here $\Gamma_\eps$ and $\Gamma_{\eps^{-1}}$ are two contours encircling the poles $\eps$ and $\eps^{-1}$ respectively, but no other poles. The contour $\Gamma_{0,\eps}$ is a contour encircling the origin and the contour $\Gamma_{\eps}$ but not $\Gamma_{\eps^{-1}}$. The contours are taken so that they do not intersect and have counter clockwise orientation. See also Figure \ref{fig:lemdeform}.
\end{lemma}
\begin{proof}
First split the contour $\Gamma_{\eps,\eps^{-1}}$ into two small contours $\Gamma_{\eps}$ around $\eps$ and $\Gamma_{\eps^{-1}}$ around $\eps^{-1}$. Then
\begin{multline}
K(x_1, m_1,x_2,m_2) = -\frac{\chi_{m_1<m_2}}{2\pi{\rm i}} \oint_{\Gamma_0}  \frac{G_{t,m_1,x_1}(z)}{ G_{t,m_2,x_2}(z)} \frac{{\rm d}z}{z}+\frac{1}{(2\pi{\rm i})^2} \oint_{\Gamma_0} \oint_{\Gamma_{\eps}} \frac{G_{t,m_1,x_1}(w)}{G_{t,m_2,x_2}(z)}\frac{1}{z(w-z)} \ {\rm d}z {\rm d}w\\
+\frac{1}{(2\pi{\rm i})^2} \oint_{\Gamma_0} \oint_{\Gamma_{\eps^{-1}}} \frac{G_{t,m_1,x_1}(w)}{G_{t,m_2,x_2}(z)}\frac{1}{z(w-z)} \ {\rm d}z {\rm d}w
\end{multline}
Now in the first double integral in the right-hand side we deform $\Gamma_{0}$ so that it encircles $\Gamma_{\eps}$. This deformation will be denoted by $\Gamma_{0,\eps}$. Note that we now pick up a residue at $w=z$. Hence by deforming we obtain an extra single integral over the contour $\Gamma_{\eps}$
\begin{multline}
K(x_1,m_1,x_2,m_2)= -\frac{\chi_{m_1<m_2}}{2\pi{\rm i}} \oint_{\Gamma_{0}} \frac{G_{t,m_1,x_1}(z)}{ G_{t,m_2,x_2}(z)} \frac{{\rm d}z}{z}-\frac{1}{2\pi{\rm i}} \oint_{\Gamma_\eps} \frac{G_{t,m_1,x_1}(z)}{ G_{t,m_2,x_2}(z)} \frac{{\rm d}z}{z}\\+
\frac{1}{(2\pi{\rm i})^2} \oint_{\Gamma_0,\eps} \oint_{\Gamma_{\eps}}  \frac{G_{t,m_1,x_1}(w)}{G_{t,m_2,x_2}(z)}\frac{1}{z(w-z)} \ {\rm d}z {\rm d}w+\frac{1}{(2\pi{\rm i})^2} \oint_{\Gamma_0} \oint_{\Gamma_{\eps^{-1}}} \frac{G_{t,m_1,x_1}(w)}{G_{t,m_2,x_2}(z)}\frac{1}{z(w-z)} \ {\rm d}z {\rm d}w
\end{multline}
Now the extra single integral has the same integrand as the first single integral. The integration is now over a contour encircling the pole at $\eps$ (and no other pole). However, this pole is not present in the case $m_1<m_2$ and then the integral vanishes. Therefore we can glue  the integrals over $\Gamma_{0}$ and $\Gamma_{\eps}$ together and obtain
\begin{multline}
K(x_1,m_1,x_2,m_2)= -\frac{\chi_{m_1<m_2}}{2\pi{\rm i}} \oint_{\Gamma_{0,\eps}} \frac{G_{t,m_1,x_1}(z)}{ G_{t,m_2,x_2}(z)} \frac{{\rm d}z}{z}+\frac{1}{(2\pi{\rm i})^2} \oint_{\Gamma_0,\eps} \oint_{\Gamma_{\eps}} \frac{G_{t,m_1,x_1}(w)}{G_{t,m_2,x_2}(z)}\frac{1}{z(w-z)} \ {\rm d}z {\rm d}w
\\+\frac{1}{(2\pi{\rm i})^2} \oint_{\Gamma_0} \oint_{\Gamma_{\eps^{-1}}}  \frac{G_{t,m_1,x_1}(w)}{G_{t,m_2,x_2}(z)}\frac{1}{z(w-z)} \ {\rm d}z {\rm d}w
\end{multline}

Finally, note that the integrand  has no pole at $w=\eps$ so in the second double integral we can safely deform $\Gamma_0$ to be $\Gamma_{0,\eps}$. This proves the claim.
\begin{figure}
\begin{center}\input{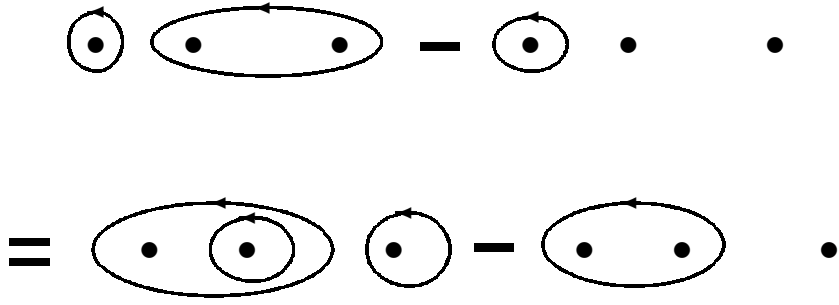}
\caption{Deforming the contours as in Lemma \ref{lem:deform}.} \label{fig:lemdeform}
\end{center}
\end{figure}
\end{proof}
Now the proof of Theorem \ref{th:Keps} follows by simply taking the limit in the integrands and  correctly choosing the contours $\Gamma_{\eps} $ and $\Gamma_{\eps^{-1}}$.

\begin{proof}[Proof of Theorem \ref{th:Keps}]

Since $\eps$ will be small, we  take the contour $\Gamma_{0,\eps}$ to be some fixed contour encircling the origin independent of $\eps$, say the unit circle for simplicity. Due to the behavior of $G_{t,m_2,x_2}(z)$ for $z\to \infty$ we can deform $\Gamma_{\eps^{-1}}$ to any simple contour $\Sigma$ that connects $\infty$ to $\infty$ crossing the real axis once at a point between $1$ and $\eps^{-1}$ and contained in the sector
\begin{align}
\{z\in \C \mid -\pi/2+\delta < \arg z <\pi/2-\delta\}.
\end{align}
for some fixed $\delta>0$. In fact we choose $\Sigma$ to be a contour so that it eventually falls inside the sector
\begin{align}
\{z\in \C \mid -\pi/2+\delta < \arg z<-\pi/4-\delta, \quad \pi/4+\delta <\arg z <\pi/2-\delta\},
\end{align}
for some $\delta>0$.

The behavior of $G_{t,m_2,x_2}(z)$ for $z\to 0$ is to a large extent similar to the behavior near $\infty$, which can be seen by performing the transform $z\mapsto 1/z$. We can  (and do) deform the contour $\Gamma_{\eps}$ to be $\Sigma^{-1}$. Note that in this way, we have deformed the contours, so that they go through the essential singularities of the integrand.

Now that we have chosen the contours, that are clearly independent of $\eps$, we can compute the limits of the integrand, which is given in \eqref{eq:limitG}. Due to the choice of the sector in which $\Sigma$ eventually falls inside, the convergence is uniform in $z\in \Sigma\cup\Sigma_{-1}$ and $w\in \Gamma_0$. This proves the statement.
\end{proof}

\subsection{Proof of Theorem \ref{th:endpoints}}

\begin{proof}
We start by noting that  by expanding a factor $(1-\eps/w)(1+\eps w)$ in the integrals for the kernel $K$ in \eqref{eq:kernel} we  get for odd $n,m$
\begin{multline}\label{eq:proof:rec1}
K(x,n+2,y,m)=(1+\eps^2)K(x,n,y,m)-\eps(K(x+1,n,y,m)+K(x-1,n,y,m))\\+\frac{\delta_{n,m-2}}{2\pi} \oint (1-\eps w)(1-\eps/w) w^{x-y-1}{\rm d}w.
\end{multline}
There is an additional  single integral because of the fact that $\xi_{n+2<m}=\xi_{n<m}-\delta_{n,m-1}-\delta_{n,m-2}$. Now since $n,m$ are both odd $\delta_{n,m-1}$ vanishes trivially and only $\delta_{n,m-2}$ remains.

Similarly,
\begin{multline}\label{eq:proof:rec2}
K(x,n,y,m-2)=(1+\eps^2)K(x,n,y,m)-\eps(K(x,n,y+1,m)+K(x,n,y-1,m))\\+\frac{\delta_{n,m-2}}{2\pi} \oint (1-\eps w)(1-\eps/w) w^{x-y-1}{\rm d}w.
\end{multline}
Therefore
\begin{multline}
-K(x_i,m_i+2,x_j,m_j)+K(x_i,m_i,x_j,m_j)=\eps\Big(K(x_i+1,m_i,x_j,m_j)+K(x_i-1,m_i,x_j,m_j) \\+\delta_{m_i,m_j-2} (\delta_{x_i+1,x_j}+\delta_{x_i-1,x_j})\Big)-\delta_{m_i,m_j-2} \delta_{x_i,x_j}+\OO(\eps^2)\end{multline}
Now note that if we have $m_i\neq m_j$ then $m_i-m_j$ is of order $L$, so that we can ignore the term $\delta_{m_i,m_j-2}$, so that
\begin{multline}\label{eq:proof:rec3}
-K(x_i,m_i+2,x_j,m_j)+K(x_i,m_i,x_j,m_j)=\eps\big(K(x_i+1,m_i,x_j,m_j)+K(x_i-1,m_i,x_j,m_j)\big)
\end{multline}
and
\begin{multline}\label{eq:proof:rec4}
-K(x_i,m_i+2,x_j,m_j+2)+K(x_i,m_i,x_j,m_j+2)=\eps\Big(K(x_i+1,m_i,x_j,m_j+2)+K(x_i-1,m_i,x_j,m_j+2)\\+\delta_{m_i,m_j} (\delta_{x_i+1,x_j}+\delta_{x_i-1,x_j})\Big)-\delta_{m_i,m_j} \delta_{x_i,x_j}+\OO(\eps^2)\end{multline}
By subtracting \eqref{eq:proof:rec4} from \eqref{eq:proof:rec3} and inserting \eqref{eq:proof:rec2} we obtain
\begin{multline}\label{eq:proof:rec5}
-K(x_i,m_i+2,x_j,m_j)+K(x_i,m_i,x_j,m_j)+K(x_i,m_i+2,x_j,m_j+2)-K(x_i,m_i,x_j,m_j+2)=\delta_{m_i,m_j}\delta_{x_i,x_j}+\OO(\eps^2)
\end{multline}
By the complementation principle (see \cite{BOO})
\begin{align}
\tilde \rho_N((x_1,m_1),\ldots,(x_N,m_N))=\det
\begin{pmatrix}
K(x_i,m_i,x_j,m_j)& -K(x_i,m_i+2,x_j,m_j)\\
K(x_i,m_i,x_j,m_j+2)&I-K(x_i,m_i+2,x_j,m_j+2)
\end{pmatrix}
\end{align}
By adding the first column to the second and afterward subtracting the second row from the first we obtain that the determinant equals
\begin{align}
\det\begin{pmatrix}
K(x_i,m_i,x_j,m_j)-K(x_i,m_i,x_j,m_j+2)& -K(x_i,m_i+2,x_j,m_j)+K(x_i,m_i,x_j,m_j)\\
&-\delta_{(x_i,m_i),(x_j,m_j)}+K(x_i,m_i+2,x_j,m_j+2)-K(x_i,m_i,x_j,m_j+2)\\\\
K(x_i,m_i,x_j,m_j+2)&\delta_{(x_i,m_i),(x_j,m_j)}-K(x_i,m_i+2,x_j,m_j+2)+K(x_i,m_i,x_j,m_j+2)
\end{pmatrix}
\end{align}
Now using, \eqref{eq:proof:rec2} in the upper left block, \eqref{eq:proof:rec5} in the upper right block, and \eqref{eq:proof:rec4} and \eqref{eq:proof:rec2} in the lower right block, this determinant has the form\begin{align}\tilde \rho_N((x_1,m_1),\ldots,(x_N,m_N))=\det\begin{pmatrix}
1+\OO(\eps^2) & \OO(\eps^2)\\
\OO(1) & \eps(K(x_i+1,m_i,x_j,m_j)+K(x_i-1,m_i,x_j,m_j))
\end{pmatrix},
\end{align}
from which the proposition easily follows.
\end{proof}
\subsection{Proof of Proposition \ref{prop:sym}}

\begin{proof}
1. The statement easily follows by  the transform $z\mapsto 1/z$ and $w\mapsto 1/w$ in the integral representation of $K^{\epsilon}(-x_1,\mu_1,-x_2,\mu_2)$.

2. The second property follows by the transformation $z\mapsto -z$ and $w\mapsto -w$ in the integrals and a deformation of $\Sigma$. In the definition of the kernel $K^{\epsilon}$ we take $\Sigma$ on the right of $\Gamma_0$. However, we can also  take $\Sigma$ to be on the left at the cost  of an extra integral as shown in Figure \ref{fig:contourswitch}.  The second double integral at the right is easy to compute since the integral over $z$  encircles the pole $z=w$ only and hence can be computed by the Residue Theorem.   The result of this is that we can rewrite the kernel in the following way
\begin{figure}
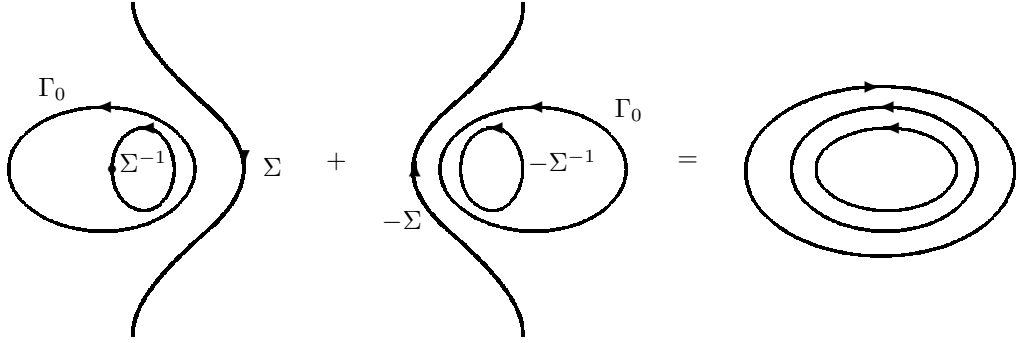

\begin{center}
\input{CrK}  \input{CrKmin} \input{CrKRes}
\end{center}
\caption{In the definition of $\mathcal K^\epsilon$ we can switch the contour $\Sigma$ to $-\Sigma$ to obtain \eqref{eq:kernelepsleft}. }\label{fig:contourswitch}
\end{figure}

\begin{align}\label{eq:kernelepsleft}
K^\epsilon(x_1,\mu_1,x_2,\mu_2)= \delta_{(x_1,\mu_1,x_2,\mu_2)} +\xi_{\mu_2<\mu_1} \int_{\Gamma_0} -\oint_{\Gamma_0} \int_{-\left(\Sigma\cup \Sigma^{-1}\right)}
\end{align}
Now by inserting $\mu_j=-\mu_j$ and the transformation $w\mapsto -w$ and $z\mapsto -z$ we arrive at the statement.\end{proof}

\subsection{Proof of Theorem \ref{th:hermite}}

\begin{proof}
By the first symmetry property in Proposition \ref{prop:sym} it suffices to consider the case $x_2<0$ only.  Using the transform $z\mapsto z/\epsilon$ and $w\mapsto w/\epsilon$ we obtain
\begin{multline}
\epsilon^{x_1-x_2} \mathcal K^{\epsilon}(x_1,\mu_1,x_2,\mu_2) =
-\chi_{\mu_1<\mu_2} \int_{\Gamma_0} {\rm e}^{(\mu_2-\mu_1)(w+\epsilon^2 /w)} w^{x_1-x_2-1}{\rm d} w
\\
+\frac{1}{(2\pi{\rm i})^2}\oint_{\Gamma_0} \int_{\Sigma\cup\Sigma^{-1}}\frac{{\rm e}^{ \mu_2 (z+\frac{\epsilon^2}{z})+\frac{1}{2}(z+\frac{\epsilon^2}{z})^2}w^{x_1}}
{{\rm e}^{ \mu_1 (w+\frac{\epsilon^2}{w})+\frac{1}{2}(w+\frac{\epsilon^2}{w})^2}z^{x_2}}
\frac{{\rm d}z {\rm d}w}{z(w-z)}.
\end{multline}
Setting $\epsilon=0$ at the right-hand side gives
\begin{multline}
\lim_{\epsilon\downarrow 0}\epsilon^{x_1-x_2} \mathcal K^{\epsilon}(x_1,\mu_1,x_2,\mu_2) =
-\chi_{\mu_1<\mu_2} \int_{\Gamma_0} {\rm e}^{(\mu_2-\mu_1)w} w^{x_1-x_2-1}{\rm d} w\\
+\frac{1}{(2\pi{\rm i})^2}\oint_{\Gamma_0} \int_{ \Sigma\cup \Sigma^{-1}}\frac{{\rm e}^{ \mu_2 z+\frac{1}{2}z^2}w^{x_1}}
{{\rm e}^{ \mu_1 w+\frac{1}{2}w^2}z^{x_2}}
\frac{{\rm d}z {\rm d}w}{z(w-z)}.
\end{multline}
The single integral can be easily computed. As for the double integral, we note that since $x_2< 0$ we have that the integral over $\Sigma^{-1}$ vanishes and hence
\begin{multline}
\lim_{\epsilon\downarrow 0}\epsilon^{x_2-x_1} \mathcal K^{\epsilon}(x_1,\mu_1,x_2,\mu_2) =
-\chi_{\mu_1<\mu_2} \int_{\Gamma_0} {\rm e}^{(\mu_2-\mu_1)w} w^{x_1-x_2-1}{\rm d} w\\
+\frac{1}{(2\pi{\rm i})^2}\oint_{\Gamma_0} \int_{ \Sigma}\frac{{\rm e}^{ \mu_2 z+\frac{1}{2}z^2}w^{x_1}}
{{\rm e}^{ \mu_1 w+\frac{1}{2}w^2}z^{x_2}}
\frac{{\rm d}z {\rm d}w}{z(w-z)}.
\end{multline}
The single integral can easily be computed. This proves \eqref{eq:kerneleps=0a}.

Finally, note that if $x_1\geq 0$ then both integrals vanish since  there is no pole for $w=0$.
\end{proof}

\subsection{Proof of Theorem \ref{th:pearcey}}
\begin{proof}

The double integral in \eqref{eq:kernelCr}  in the new parameters given by \eqref{eq:newparamP} reads
\begin{align}
\frac{1}{(2\pi {\rm i})^2} \oint_{\Gamma_{0}} {\rm d}w \int_{\Sigma\cup \Sigma^{-1}}{\rm d}z\  \frac
{{\rm e}^{\frac{M^2}{2} (z-1)^4/z^2+M \tilde \nu_2 (z+1/z)-M^{1/2} \xi_1 \ln z}}
{{\rm e}^{\frac{M^2}{2} (w-1)^4/z^2+M\tilde \nu_2 (w+1/w)-M^{1/2} \xi_2  \ln z}}
\frac{1}{z(w-z)}
\end{align}
For large $M$ the main contribution comes from the terms
\begin{align}
\frac{(z-1)^4}{z^2}\quad \textrm{ and } \quad \frac{(w-1)^4}{w^2}.
\end{align}
\begin{figure}[t]
\begin{center}
\includegraphics[scale=0.3]{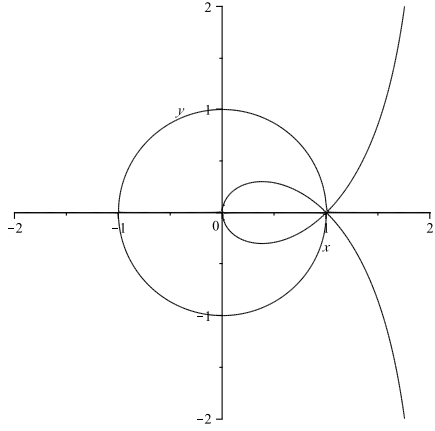}
\end{center}
\caption{The line $\{w\mid \Im((w-1)^4/w^2)=0\}$ which are the contours of steepest descent/ascent leaving from $w=1$. }\label{fig:pathsstds}
\end{figure}
This term has a critical point at $z=1$  (resp. $w=1$) of order three. Therefore, there are four paths of steepest descent leaving from $z=1$ and four paths of steepest ascend, as shown in Figure \ref{fig:pathsstds}. In fact, the paths of steepest ascent leaving from $w=1$ are the unit circle and the real line.  We now deform $\Gamma_0$ to be the unit circle. The contour $\Sigma$ is deformed so that it passes through $z=1$ and follows the path of steep descent outside the unit circle. Then locally around $z=1$ and $w=1$ the contours $\Gamma_0$,  $\Sigma$ and $\Sigma^{-1}$ can be deformed to the contours for the Pearcey kernel  as shown in Figure \ref{fig:pearcey}.

By standard steepest descent arguments one can now show that the dominant contribution comes from a neighborhood around $z=1$ and $w=1$. More precisely, after introducing the local variables \begin{align}
\left\{
\begin{array}{l}
w=1+\tilde w/M^{1/2}\\
z=1+\tilde z/M^{1/2}
\end{array}
\right.
\end{align}
it is not difficult to prove that
\begin{multline}\label{eq:proofPdouble}
\frac{1}{(2\pi {\rm i})^2} \oint_{\Gamma_{0}} {\rm d}w \int_{\Sigma\cup \Sigma^{-1}}{\rm d}z\  \frac
{{\rm e}^{\frac{M^2}{2} (z-1)^4/z^2+M \nu_2 (z+1/z)-M^{1/2} \xi_2 \ln z}}
{{\rm e}^{\frac{M^2}{2} (w-1)^4/z^2+M \nu_1 (w+1/w)-M^{1/2} \xi_1 \ln z}}
\frac{1}{z(w-z)}\\
=\frac{M^{1/2}{\rm e}^{2M(\nu_2-\mu_1)}}{(2\pi{\rm i})^2} \int_{-{\rm i}\infty}^{{\rm i}\infty}{\rm d}z \int_{\mathcal C} {\rm d}w \
\frac
{{\rm e}^{\frac{1}{2} \tilde z^4+ \nu_2 \tilde z^2- \xi_2 z }}
{{\rm e}^{\frac{1}{2}  \tilde w^4+\nu_1 \tilde w^2- \xi_1 w}}
\frac{1}{w-z}\left(1+\OO(M^{-1/2})\right),
\end{multline}
as $M\to \infty$.

Now consider the single integral in \eqref{eq:kernelCr}, which in the new  parameters \eqref{eq:newparamP} reads
\begin{align}
\frac{1}{2\pi{\rm i}}\oint_{\Gamma_0} {\rm d}w \ {\rm e}^{M(\nu_2-\nu_1)(w+1/w)}w^{M^{1/2} (\xi_1-\xi_2)-1}
\end{align}
The dominant term in this integral is
\begin{align}
{\rm e}^{M(\nu_2-\nu_1) (w+1/w)}.
\end{align}
A simple computation shows that $w+1/w$ has two saddle points $w=\pm 1$, both of order one. Since $\nu_2>\nu_1$ (otherwise the single integral is not present) we have that $w=+1$ is the dominant saddle point. This means that, the main contribution comes from the part of the contour close to $w=1$. Hence, if we introduce the new local variable
\begin{align}
w=1+\tilde w/M^{1/2},
\end{align}
then by standard arguments one can prove that
\begin{multline} \label{eq:proofPsingle}
\frac{1}{2\pi{\rm i}}\oint_{\Gamma_0} {\rm d}w \ {\rm e}^{M(\nu_2-\nu_1)(w+1/w)}w^{M^{1/2} (\xi_1-\xi_2)-1}
\\
=\frac{M^{1/2} {\rm e}^{2M(\nu_2-\nu_1)}}{2\pi{\rm i}}\int_{-{\rm i}\infty}^{{\rm i}\infty}{\rm d}\tilde w \ {\rm e}^{(\nu_2-\nu_1)\tilde w^2-(\xi_1-\xi_2)\tilde w} \left(1+\OO(M^{-1/2})\right)
\end{multline}
By inserting \eqref{eq:proofPsingle} and \eqref{eq:proofPdouble} in \eqref{eq:kernelCr} and taking the limit $M\to \infty$ we obtain \eqref{eq:kernelP}.
This proves Theorem \ref{th:pearcey}.
\end{proof}

\end{document}